\documentclass[11pt]{amsart}
\usepackage{amssymb,amsmath}
\usepackage{hyperref}
\usepackage{bm}
\usepackage[]{graphicx}
\usepackage{amssymb, epsfig}
\numberwithin{equation}{section}

\numberwithin{figure}{section}
\newcommand{\eps}{\varepsilon}

\newtheorem{theorem}{Theorem}[section]

\newtheorem{remark}[theorem]{Remark}
\begin{document}
%
%
%
%
\title[Addendum with non-zero initial data]{Addendum to the article "On the Dirichlet to Neumann Problem for the 1-dimensional
Cubic NLS Equation on the Half-Line"}
\author[D.~C. Antonopoulou]{D.~C. Antonopoulou$^{\$}$}

\author[S. Kamvissis]{S. Kamvissis$^{*}$}
\thanks{$^*$ Department of Pure and Applied Mathematics,
University of Crete, GR--700 13 Heraklion, Greece, and, Institute
of Applied and Computational Mathematics, FORTH, GR--711 10
Heraklion, Greece, email: spyros@tem.uoc.gr}
\thanks{$^\$$ Department of Mathematics, University
of Chester, Thornton Science Park, CH2 4NU, UK, and, Institute of
Applied and Computational Mathematics, FORTH, GR--711 10
Heraklion, Greece, email: d.antonopoulou@chester.ac.uk}

%
%
\begin{abstract}
We present a short note on the extension of the results of
\cite{AKnon} to the case of non-zero initial data. More
specifically, the defocusing cubic NLS equation is considered on
the half-line with decaying (in time) Dirichlet data  and
sufficiently smooth and decaying (in space) initial data. We prove that for
this case also, and for a large class of decaying Dirichlet data,
the Neumann data are  sufficiently decaying so that the Fokas unified
method for the solution of defocusing NLS is applicable.
%
\end{abstract}
%
%
\maketitle
\pagestyle{myheadings}
\thispagestyle{plain}
%
%
%
\section{Introduction}
In this note, we consider the defocusing non-linear  Schr\"odinger
equation (NLS) with cubic non-linearity, posed on the real
positive semi-axis $\mathbb{R}^+$
\begin{equation}\label{nls}
{\rm i}q_t+q_{xx}-2|q|^2q=0,\;\;\;\;x> 0,\;\;\;\;0< t<
+\infty,
\end{equation}
and initial-boundary data
\begin{equation}\label{ic}
\begin{split}
&q(x,0)=q_0(x),\;\;\;\;0\leq x<+\infty\\
&q(0,t)=Q(t),\;\;\;\;0\leq t<+\infty,\\
\end{split}
\end{equation}
where  $q_0, Q$ are classical functions satisfying
the compatibility condition $q_0(0)=Q(0)$.

Existence and uniqueness of solution for the problem
\eqref{nls}-\eqref{ic} with $q_0\in H^2$, $Q\in C^2$ and
$q_0(0)=Q(0)$, have been established in \cite{CarBu}.

Our aim here is to show that under specific conditions on $q_0$ and $Q$, the
function $q_x(0,t)$ is in $L^1(dt)$ and hence the initial/boundary value problem
admits a full analysis via the unified inverse scattering method initiated and studied by Fokas and collaborators.

In a previous publication \cite{AKnon}
we considered the simplified problem with zero initial data.
It turns out that the result  can be easily extended to the case of general initial data
satisfying  particular decay and smoothness conditions. For the convenience of the reader we provide here a
self-contained proof with full details.

\section{An $L^2(0,\infty)$ estimate for the Neumann data}

We begin with some notations: $(\cdot,\cdot)$ will denote the
$L^2(0,\infty)$ inner product in space variables, and $\|\cdot\|$
the induced norm. The symbol $\|\cdot\|_4$ is used for the
$L^4(0,\infty)$ norm in space, while for any integer $p\geq 1$,
$\|\cdot\|_{L^p(0,t)}$ will denote the $L^p$ norm in the time
interval $(0,t)$. The symbol $c$ will be used to denote a generic positive
constant.

We emphasize here the following decay condition for the solution $q$ of
the NLS problem \eqref{nls}-\eqref{ic}
\begin{equation}\label{dc}
q(x,t)\rightarrow
0\;\;\mbox{as}\;\;\;\;x\rightarrow\infty\;\;\;\;\mbox{for
any}\;\;t\geq 0,
\end{equation}
which is crucial for our arguments, and is of course
provided by the result  in \cite{CarBu}.

\begin{theorem}\label{thm1} Let $q$ be the
unique global classical solution $q\in C^1(L^2)\cap C^0(H^2)$ of
the problem \eqref{nls}-\eqref{ic}, with  $Q\in C^2$ and
$Q(0)=q_0(0)$. It holds
that for any $t>0$
\begin{equation}\label{mainineq1}
\begin{split}
\Big{(}\int_0^t|q_x(0,s)|^2ds\Big{)}^{1/2}=&\|q_x(0,\cdot)\|_{L^2(0,t)}\\
\leq&
c\|q(0,\cdot)\|_{L^2(0,t)}^2+c\|q_t(0,\cdot)\|_{L^2(0,t)}^2+\|q(0,\cdot)\|_{L^4(0,t)}^4\\
&+c\|q_0\|^2+c\|q_{0x}\|^2+c\|q_0\|_{L^4(0,\infty)}^4,
\end{split}
\end{equation}
where $c$ is a positive constant. The previous estimate is true also
when $t$ is replaced by $\infty$.
\end{theorem}
\begin{proof}
Multiplying \eqref{nls} by $\bar{q}$ and integrating in
$x\in(0,\infty)$ we obtain after applying integration by parts
$${\rm
i}(q_t,q)-\|q_x\|^2-q_x(0,t)\bar{q}(0,t)-2(|q|^2q,q)=0.$$
Taking imaginary parts we arrive at
\begin{equation}\label{j**}
\frac{d}{dt}\|q\|^2=2{\rm Im}\{q_x(0,t)\bar{q}(0,t)\},
\end{equation}
 which after
integration in time gives
\begin{equation}\label{mt2}
\|q(\cdot,t)\|^2\leq
\|q(\cdot,0)\|^2+2\|q_x(0,\cdot)\|_{L^2(0,t)}\|q(0,\cdot)\|_{L^2(0,t)}.
\end{equation}

Multiplying now \eqref{nls} by $\bar{q_t}$ and integrating in
$x\in(0,\infty)$ we obtain
$${\rm i}
\|q_t\|^2-(q_x,q_{xt})-q_x(0,t)\bar{q_t}(0,t)-2(|q|^2q,q_t)=0,$$
while taking real parts and since
$$\frac{d}{dt}[|q|^2|q|^2]=4{\rm Re}\{|q|^2q\bar{q_t}\},$$ we have
\begin{equation}\label{j0}
\frac{d}{dt}\|q_x\|^2=-\frac{d}{dt}\||q|^2\|^2-2{\rm
Re}\{q_x(0,t)\bar{q_t}(0,t)\}.
\end{equation}
 Integration in time yields
$$\|q_x(\cdot,t)\|^2-\|q_x(\cdot,0)\|^2=
-\||q(\cdot,t)|^2\|^2+\||q(\cdot,0)|^2\|^2-2\int_0^t{\rm Re}\{q_x(0,s)\bar{q_t}(0,s)\}ds,$$
and thus,
\begin{equation}\label{mt3}
\|q_x(\cdot,t)\|^2+\|q(\cdot,t)\|_{L^4(0,\infty)}^4\leq
\|q_x(\cdot,0)\|^2+\|q(\cdot,0)\|_{L^4(0,\infty)}^4+
2\|q_x(0,\cdot)\|_{L^2(0,t)}\|q_t(0,\cdot)\|_{L^2(0,t)}.
\end{equation}

We now multiply \eqref{nls} with $\bar{q_x}$ integrate in space
and take real parts which yields
$$-{\rm Im}(q_t,q_x)+{\rm Re}(q_{xx},q_x)-2{\rm
Re}(|q|^2q,q_x)=0.$$ Using that
$${\rm Im}(q_t,q_x)=-\frac{1}{2}{\rm
i}\frac{d}{dt}(q,q_x)-\frac{1}{2}{\rm i}q(0,t)\bar{q_t}(0,t),$$
and the relations
$${\rm Re}(q_{xx},q_x)=-\frac{1}{2}|q_x(0,t)|^2,\;\;\;\;{\rm
Re}(|q|^2q,q_x)=-\frac{1}{4}|q(0,t)|^4,$$ we obtain
\begin{equation}\label{ja1}
|q_x(0,t)|^2={\rm i}\frac{d}{dt}(q,q_x)+{\rm
i}q(0,t)\bar{q_t}(0,t)+|q(0,t)|^4.
\end{equation}
 Integrating the above
in time, we get
$$\int_0^t|q_x(0,s)|^2ds={\rm i}(q(\cdot,t),q_x(\cdot,t))-{\rm i}(q(\cdot,0),q_x(\cdot,0))+{\rm
i}\int_0^tq(0,s)\bar{q_t}(0,s)ds+\int_0^t|q(0,s)|^4ds,$$
and so,
\begin{equation}\label{mt4}
\begin{split}
\|q_x(0,\cdot)\|_{L^2(0,t)}^2\leq&\|q(\cdot,t)\|\|q_x(\cdot,t)\|+\|q(\cdot,0)\|\|q_x(\cdot,0)\|\\
&+\|q(0,\cdot)\|_{L^2(0,t)}\|q_t(0,\cdot)\|_{L^2(0,t)}+\|q(0,\cdot)\|_{L^4(0,t)}^4\\
\leq &c\|q(\cdot,t)\|^2+c\|q_x(\cdot,t)\|^2+\|q(\cdot,0)\|\|q_x(\cdot,0)\|\\
&+\|q(0,\cdot)\|_{L^2(0,t)}\|q_t(0,\cdot)\|_{L^2(0,t)}+\|q(0,\cdot)\|_{L^4(0,t)}^4.
\end{split}
\end{equation}
Using now \eqref{mt2} and \eqref{mt3}  we get
\begin{equation}\label{mt4*}
\begin{split}
\|q_x(0,\cdot)\|_{L^2(0,t)}^2\leq
&c\|q(\cdot,0)\|^2+2c\|q_x(0,\cdot)\|_{L^2(0,t)}\|q(0,\cdot)\|_{L^2(0,t)}\\
&+c\|q_x(\cdot,0)\|^2+c\|q(\cdot,0)\|_{L^4(0,\infty)}^4+
2c\|q_x(0,\cdot)\|_{L^2(0,t)}\|q_t(0,\cdot)\|_{L^2(0,t)}\\
&+\|q(\cdot,0)\|\|q_x(\cdot,0)\|
+\|q(0,\cdot)\|_{L^2(0,t)}\|q_t(0,\cdot)\|_{L^2(0,t)}+\|q(0,\cdot)\|_{L^4(0,t)}^4\\
\leq
&c_0\|q_x(0,\cdot)\|_{L^2(0,t)}^2+c\|q(0,\cdot)\|_{L^2(0,t)}^2+c\|q_t(0,\cdot)\|_{L^2(0,t)}^2\\
&+c\|q(\cdot,0)\|^2+c\|q_x(\cdot,0)\|^2+c\|q(\cdot,0)\|_{L^4(0,\infty)}^4\\
&+\|q(\cdot,0)\|\|q_x(\cdot,0)\|+\|q(0,\cdot)\|_{L^2(0,t)}\|q_t(0,\cdot)\|_{L^2(0,t)}+\|q(0,\cdot)\|_{L^4(0,t)}^4,
\end{split}
\end{equation}
where $c_0$ can be made as small as we want. Thus, we arrive at
the result \eqref{mainineq1}, i.e.
\begin{equation*}
\begin{split}
\|q_x(0,\cdot)\|_{L^2(0,t)}^2\leq&
c\|q(0,\cdot)\|_{L^2(0,t)}^2+c\|q_t(0,\cdot)\|_{L^2(0,t)}^2+\|q(0,\cdot)\|_{L^4(0,t)}^4\\
&+c\|q(\cdot,0)\|^2+c\|q_x(\cdot,0)\|^2+c\|q(\cdot,0)\|_{L^4(0,\infty)}^4.
\end{split}
\end{equation*}
\end{proof}
\begin{remark}\label{rem1}
Note that the previous estimate provides a bound for
$$\|q_x(0,\cdot)\|_{L^2(0,\infty)}^2,$$ when $q_0$ is sufficiently
smooth.

More specifically, if $q_0\in H^1(0,\infty)\cap L^4(0,\infty)$ and
if $q(0,t)$, $q_t(0,t)$, have polynomial decay of order
$\mathcal{O}(t^{-\alpha})$, $\mathcal{O}(t^{-\beta})$
respectively as $ t \to \infty,$ with $\alpha>1/2$ and $\beta>1/2$
then the  estimate \eqref{mainineq1} of Theorem \ref{thm1}
implies
\begin{equation}\label{unifqx}\int_0^\infty
|q_x(0,t)|^2dt<\infty.
\end{equation}

In addition, the estimate \eqref{unifqx}, together with
\eqref{mt2} and \eqref{mt3} yield that there exists $c>0$
independent of $t$, such that for any $t\geq 0$
\begin{equation}\label{j***}
\|q(\cdot,t)\|\leq c,
\end{equation}
\begin{equation}\label{j***1}
\|q_x(\cdot,t)\|\leq c,
\end{equation}
and
\begin{equation}\label{jn1}
\|q(\cdot,t)\|_{L^4(0,\infty)}\leq c.
\end{equation}
\end{remark}

\section{Decay of solution as $t\rightarrow \infty$}

In the sequel, we prove, under certain assumptions for the initial
data, that the solution of \eqref{nls}-\eqref{ic} decays to $0$
for any $x$ as $t\rightarrow \infty$. For this, we first establish that
the $L^4$ norm of the solution decays like
$\mathcal{O}\Big{(}t^{-\frac{1}{4}}\Big{)}$ as $t\rightarrow
\infty$.

\begin{theorem}\label{thm2} Let $q$ be the
unique global classical solution $q\in C^1(L^2)\cap C^0(H^2)$ of
the problem \eqref{nls}-\eqref{ic}, with  $Q\in C^2$ and
$Q(0)=q_0(0)$.

Assume that  $x q_0\in L^2(0,\infty)$, so in particular
$q_0\in H^1(0,\infty)\cap L^4(0,\infty)$. Furthermore, assume that as
$t\rightarrow\infty$
$$q(0,t)= \mathcal{O}(t^{-\alpha}),\;\;\;\;q_t(0,t)=
\mathcal{O}(t^{-\beta}),\;\;\;\;\mbox{for}\;\;\alpha>
3/2\;\;\;\mbox{and}\;\;\beta> 5/2.$$

It holds that there exists a positive constant $c$ independent of
$t$ such that
\begin{equation}\label{j2}
\int_0^\infty|q(x,t)|^4dx:=\|q(\cdot,t)\|_4^4\leq
\frac{c}{t}\;\;\;\;\mbox{for any}\;\;t\geq 1.
\end{equation}
\end{theorem}
\begin{proof}
Again let the Dirichlet data be $$Q(t):=q(0,t),$$ and
let us denote the Neumann data by $$P(t):=q_x(0,t).$$

We set
$$y(t)={\rm Im}\int_0^\infty x\bar{q}(x,t)q_x(x,t)dx,$$
to obtain (cf. the analytical proof of the specific identity in
\cite{AKnon})
\begin{equation*}
\begin{split}
\partial_t\Big{(}t^2\int_0^\infty|q|^4dx\Big{)}=&\frac{1}{4}\partial_t\Big{[}4ty-\int_0^\infty
x^2|q(x,t)|^2dx-4t^2\int_0^\infty|q_x|^2dx\Big{]}\\
&-t{\rm
Re}\{P\bar{Q}\}-2t^2{\rm
Re}\{P\bar{Q_t}\}+t\int_0^\infty|q|^4dx.
\end{split}
\end{equation*}
Integrating the above in time in the interval $(0,t)$, we get
\begin{equation}\label{j19}
\begin{split}
t^2\|q(\cdot,t)\|_4^4=&\frac{1}{4}\Big{[}4ty-\int_0^\infty
x^2|q(x,t)|^2dx+\int_0^\infty
x^2|q(x,0)|^2dx-4t^2\int_0^\infty|q_x(x,t)|^2dx\Big{]}\\
&-\int_0^tr{\rm Re}\{P(r)\bar{Q}(r)\}dr-
2\int_0^tr^2{\rm
Re}\{P(r)\bar{Q_r}(r)\}dr+\int_0^tr\|q(\cdot,r)\|_4^4dr.
\end{split}
\end{equation}
Observe now that
\begin{equation}\label{j20}
4ty=\int_0^\infty[x^2|q|^2+4t^2|q_x|^2-|xq+2{\rm i}tq_x|^2]dx,
\end{equation}
while by \eqref{j19} we have
\begin{equation}\label{j21}
\begin{split}
t^2\|q(\cdot,t)\|_4^4=\frac{1}{4}\Big{[}&\int_0^\infty(x^2|q(x,t)|^2+4t^2|q_x(x,t)|^2-|xq+2{\rm
i}tq_x(x,t)|^2-x^2|q(x,t)|^2+x^2|q(x,0)|^2)dx\\
&-4t^2\int_0^\infty|q_x(x,t)|^2dx\Big{]}\\
-&\int_0^tr{\rm
Re}\{P(r)\bar{Q}(r)\}dr-2\int_0^tr^2{\rm
Re}\{P(r)\bar{Q_r}(r)\}dr+\int_0^tr\|q(\cdot,r)\|_4^4dr\\
=&\frac{1}{4}\Big{[}-\int_0^\infty|xq(x,t)+2{\rm
i}tq_x(x,t)|^2dx\Big{]}+\frac{1}{4}\int_0^\infty x^2|q(x,0)|^2dx\\
&-\int_0^t r{\rm Re}\{P(r)\bar{Q}(r)\}dr-2\int_0^tr^2{\rm
Re}\{P(r)\bar{Q_r}(r)\}dr+\int_0^tr\|q(\cdot,r)\|_4^4dr.
\end{split}
\end{equation}
Hence, \eqref{j21} gives
\begin{equation}\label{j22}
t^2\|q(\cdot,t)\|_4^4\leq \int_0^t
r\|q(\cdot,r)\|_4^4dr+\frac{1}{4}\int_0^\infty
x^2|q(x,0)|^2dx+F(P,Q,Q_t,t),
\end{equation}
for
$$F(P,Q,Q_t,t):=-\int_0^t r{\rm
Re}\{P(r)\bar{Q}(r)\}dr-2\int_0^t r^2{\rm
Re}\{P(r)\bar{Q_r}(r)\}dr.$$

Next, we prove that there exists $c>0$ independent of
$t$ such that for any $t\geq 0$,
\begin{equation}\label{j25}
|F(P,Q,Q_t,t)|\leq c.
\end{equation}
Indeed,  if $a>3/2$ and $\beta>5/2$, then we have for
$F=F(P(t),Q(t),Q_t(t),t)$
\begin{equation*}
\begin{split}
|F|\leq &c\Big{(}\int_0^t|P(r)|^2dr\Big{)}^{1/2}\Big{(}\int_0^t
r^2|Q(r)|^2dr\Big{)}^{1/2}+c\Big{(}\int_0^t|P(r)|^2dr\Big{)}^{1/2}\Big{(}\int_0^tr^4|Q_r(r)|^2dr\Big{)}^{1/2}\\
\leq & c\Big{(}\int_0^t|P(r)|^2dr\Big{)}^{1/2}\leq c,
\end{split}
\end{equation*}
since \eqref{unifqx} is true. Thus, \eqref{j25} follows.

Since $$\int_0^\infty x^2|q(x,0)|^2dx<\infty,$$ then using
\eqref{j25} at \eqref{j22}, we obtain that there exists $c>0$
independent of $t$ such that
\begin{equation*}\label{j26}
t^2\|q(\cdot,t)\|_4^4\leq \int_0^t r\|q(\cdot,r)\|_4^4dr+c.
\end{equation*}
So, since $\int_0^1\|q(\cdot,t)\|_4^4dt$ is bounded as \eqref{jn1} is true, we obtain
\eqref{j2}, cf. also the analytical proof of \cite{AKnon}.
\end{proof}

\begin{remark}\label{remn2}
An immediate outcome of the last theorem
(as well as its counterpart, Theorem 3.1 in \cite{AKnon}), is
the absence of any solitons for the defocusing NLS on the half-line. This result is arrived at
without the need of employing the unified theory of Fokas. Of course, it does follow from
that theory that there are no solitons for this problem (see \cite{L}) and even more one has a complete
rigorous asymptotic description of the decaying behavior  by analyzing the Riemann-Hilbert problem asymptotically for
large times (see \cite{FIS}). But one has to remember that the above theorem is
essential in providing a class of data admissible for the unified theory to begin with.
It is thus indispensible in  justifying the use
of the unified theory, which is a prerequisite for the Riemann-Hilbert formulation.
\end{remark}

Now we are ready to prove the next main theorem, which establishes the decay of defocusing NLS solution as $t\rightarrow \infty$, uniformly in space.
\begin{theorem}\label{prop1}
Let $q$ be the
unique global classical solution $q\in C^1(L^2)\cap C^0(H^2)$ of
the problem \eqref{nls}-\eqref{ic}, with  $Q\in C^2$ and
$Q(0)=q_0(0)$.

Assume that  $x q_0\in L^2(0,\infty)$, so in particular
$q_0\in H^1(0,\infty)\cap L^4(0,\infty)$. Furthermore assume that as
$t\rightarrow\infty$
$$q(0,t)= \mathcal{O}(t^{-\alpha}),\;\;\;\;q_t(0,t)=
\mathcal{O}(t^{-\beta}),\;\;\;\;\mbox{for}\;\;\alpha>
3/2\;\;\;\mbox{and}\;\;\beta> 5/2.$$

It holds that
\begin{equation}\label{j27}
\displaystyle{\lim_{t\rightarrow\infty}}q(x,t)=0,
\end{equation}
for any $x>0$.
\end{theorem}
\begin{proof}
Indeed, we have by Theorem \ref{thm2}
$$\displaystyle{\lim_{t\rightarrow\infty}}\|q(\cdot,t)\|_4\leq
\displaystyle{\lim_{t\rightarrow\infty}}\frac{c}{t^{1/4}}=0.$$
Using that
$$\frac{d}{dx}q^3(x,t)=3q^2(x,t)q_x(x,t),$$
we have
\begin{equation}\label{apo1}
\begin{split}
|q^3(x,t)|=&|q^3(0,t)+\int_0^x3q^2(y,t)q_x(y,t)dy|\\
\leq &|q(0,t)|^3+c\Big{(}\int_0^\infty
|q(y,t)|^4dy\Big{)}^{1/2}\Big{(}\int_0^\infty
|q_x(y,t)|^2dy\Big{)}^{1/2}.
\end{split}
\end{equation}
So, under the assumption of decaying initial data
$q(0,t)\rightarrow 0$ as $t\rightarrow\infty$,
and since by Remark
\ref{rem1}
 we have
(cf. \eqref{j***1})
$$
\|q_x(\cdot,t)\|\leq c,
$$
we conclude from  \eqref{apo1} that
\begin{equation}\label{apo2}
\begin{split}
\displaystyle{\lim_{t\rightarrow\infty}}|q^3(x,t)|
 \leq
c\displaystyle{\lim_{t\rightarrow\infty}}\|q(\cdot,t)\|_4^{2}=0.
\end{split}
\end{equation}
\end{proof}

\section{An $L^1(0,\infty)$ estimate for the Neumann data}

At this Section, we shall use the result of Theorem \ref{prop1}, i.e. that
for all positive $x$,
\begin{equation*}
\displaystyle{\lim_{t\rightarrow\infty}}q(x,t)=0,
\end{equation*}
in order to prove, under the assumptions of Theorem \ref{prop1} of course, an
$L^1(0,\infty)$ bound for $q_x(0,t)$.

Consider relation \eqref{ja1}
\begin{equation*}\label{j1}
|q_x(0,t)|^2={\rm i}\frac{d}{dt}(q,q_x)+{\rm
i}q(0,t)\bar{q_t}(0,t)+|q(0,t)|^4.
\end{equation*}
and multiply with $t^p,~~p>1$ to obtain
\begin{equation}\label{ap1}
t^p|q_x(0,t)|^2={\rm i}t^p\frac{d}{dt}(q,q_x)+{\rm
i}t^pq(0,t)\bar{q_t}(0,t)+t^p|q(0,t)|^4.
\end{equation}
Integration of \eqref{ap1} in time gives
\begin{equation}\label{ap2}
\begin{split}
\int_0^\infty t^p|q_x(0,t)|^2dt
= &{\rm i} [t^p(q(\cdot,t),q_x(\cdot,t))]_0^\infty-{\rm i}\int_0^\infty
pt^{p-1}(q(\cdot,t),q_x(\cdot,t))dt\\
&+{\rm i}\int_0^\infty t^pq(0,t)\bar{q_t}(0,t)dt+
\int_0^\infty t^p|q(0,t)|^4dt\\
\leq&
\displaystyle{\lim_{t\rightarrow\infty}}\Big{(}t^p\|q(\cdot,t)\|\|q_x(\cdot,t)\|\Big{)}+c\int_0^\infty
t^{p-1}\|q(\cdot,t)\|\|q_x(\cdot,t)\|dt\\
&+c\int_0^\infty t^p|q(0,t)||q_t(0,t)|dt+ \int_0^\infty
t^p|q(0,t)|^4dt.
\end{split}
\end{equation}
As  before, integrating \eqref{j**} in $(t,\infty)$ and
using that $\|q(\cdot,\infty)\|=0$, we get
$$\|q(\cdot,t)\|^2=-2\int_t^\infty{\rm Im}\{q_x(0,r)\bar{q}(0,r)\}dr,$$
which, under the assumptions of Theorem \ref{thm2}, by
\eqref{unifqx} yields for any $t\geq 0$
\begin{equation*}
\begin{split}
\|q(\cdot,t)\|^2\leq
c\Big{(}\int_t^\infty|q(0,r)|^2dr\Big{)}^{1/2}.
\end{split}
\end{equation*}
Thus, we obtain by \eqref{ap2}

\begin{equation}\label{ap22}
\begin{split}
\int_0^\infty t^p|q_x(0,t)|^2dt\leq&
c\displaystyle{\lim_{t\rightarrow\infty}}\Big{(}t^p\Big{(}\int_t^\infty|q(0,r)|^2dr\Big{)}^{1/4}\Big{)}
\\&+c\int_0^\infty
t^{p-1}\Big{(}\int_t^\infty|q(0,r)|^2dr\Big{)}^{1/4}dt\\
&+c\int_0^\infty t^p|q(0,t)||q_t(0,t)|dt+ \int_0^\infty
t^p|q(0,t)|^4dt.
\end{split}
\end{equation}

Furthermore, since $\|q_x(\cdot,t)\|$ is bounded
$$t^p\|q(\cdot,t)\|\|q_x(\cdot,t)\|\leq
ct^p\Big{(}\int_t^\infty|q(0,r)|^2dr\Big{)}^{1/4}\rightarrow 0$$
if $q(0,t)$ is assured to have a sufficiently fast  (polynomial) decay as
$t\rightarrow\infty$ ($\mathcal{O}(t^{-\alpha})$ with $\alpha$ to be specified). 
Also, we have
$$\int_0^\infty t^{p-1}\|q(\cdot,t)\|\|q_x(\cdot,t)\|dt\leq
c\int_0^\infty
t^{p-1}\Big{(}\int_t^\infty|q(0,r)|^2dr\Big{)}^{1/4}dt\leq c,$$ again if
 $q(0,t)$ has a sufficiently fast decay as
$t\rightarrow\infty$. The same argument of sufficiently fast
decay for $q_t(0,t)$ as $t\rightarrow\infty$, (like
$\mathcal{O}(t^{-\beta})$ with $\beta$ to be specified), together with the previous one,
finally gives, using \eqref{ap22}
\begin{equation}\label{ap3}
\begin{split}
\int_0^\infty t^p|q_x(0,t)|^2dt\leq c,
\end{split}
\end{equation}
if $p>1$,  $\alpha>3/2$ and $\beta>5/2$, satisfying also
$$\alpha>2p+1/2,\;\;\;\alpha+\beta>p+1,\;\;\;\alpha>(p+1)/4.$$

Now, we are ready to  derive the  $L^1(0,\infty)$ estimate for
$|q_x(0,t)|$. We have the following  main theorem.
\begin{theorem}
Let $q$ be the
unique global classical solution $q\in C^1(L^2)\cap C^0(H^2)$ of
the problem \eqref{nls}-\eqref{ic}, with  $Q\in C^2$ and
$Q(0)=q_0(0)$.

Assume that  $x q_0\in L^2(0,\infty)$.

If $q(0,t)$, $q_t(0,t)$ have a sufficiently fast decay as
$t\rightarrow\infty$,  that is $\mathcal{O}(t^{-\alpha})$ and
$\mathcal{O}(t^{-\beta})$, for $\alpha> 5/2$ and $\beta> 5/2$
respectively, then
$$\int_0^\infty |q_x(0,t)|dt<\infty.$$
\end{theorem}
\begin{proof}
\begin{equation}\label{ap4}
\begin{split}
\int_0^\infty |q_x(0,t)|dt&=\int_0^1|q_x(0,t)|dt+\int_1^\infty
t^{-\frac{1}{2}-\frac{\eps}{2}}t^{\frac{1}{2}+\frac{\eps}{2}}|q_x(0,t)|dt\\
&\leq
c\Big{(}\int_0^11^2dt\Big{)}^{1/2}\Big{(}\int_0^1|q_x(0,t)|^2dt\Big{)}^{1/2}+
\Big{(}\int_1^\infty
t^{-1-\eps}dt\Big{)}^{1/2}\Big{(}\int_1^\infty
t^{1+\eps}|q_x(0,t)|^2dt\Big{)}^{1/2}\\
&\leq c+ c\Big{(}\int_0^\infty
t^{1+\eps}|q_x(0,t)|^2dt\Big{)}^{1/2}\leq c,
\end{split}
\end{equation}
where we used \eqref{ap3} for $p:=1+\eps$ with $\eps>0$ as small.
Note that here, indeed $\alpha>5/2>3/2$ and $\beta>5/2$ and
satisfy also
$$\alpha>2p+1/2,\;\;\;\alpha+\beta>p+1,\;\;\;\alpha>(p+1)/4.$$
\end{proof}

\begin{remark}\label{remn3}
An immediate outcome of the last theorem is that if the initial data satisfy the Carrol-Bu condition
$q_0 \in H^2$ and the extra condition $x q_0\in L^2(0,\infty)$ and if the Dirichlet data $Q \in C^2$ are such that 
$Q(0)=q_0(0)$ and
$Q(t)$, $Q_t(t)$ have a sufficiently fast decay as
$t\rightarrow\infty$,  that is $\mathcal{O}(t^{-5/2 + \epsilon})$ 
for some $\epsilon >0,$ then the unified method of Fokas $can$ be applied, even for $t \to \infty$,
and hence the resulting Riemann-Hilbert factorisation problem which sums up the
inverse "Fokas scattering transform" is valid for all time and can be used to extract long time asympotics.
\end{remark}

\section{Acknowledgement}
Research funded by ARISTEIA II grant no. 3964 from
the General Secretariat of Research and Technology, Greece.

\end{document}